\documentclass[12pt,a4paper]{amsart}
\usepackage{amsmath, amsthm, amssymb,,verbatim}
\usepackage{graphicx,amssymb,amsfonts,epsfig,amsthm,a4,amsmath,url,graphicx,enumerate}
\usepackage[latin1]{inputenc}

\usepackage[usenames,dvipsnames]{color}

\newtheorem{thm}{Theorem}[section]
\newtheorem{cor}[thm]{Corollary}
\newtheorem{lem}[thm]{Lemma}

\newtheorem{prop}[thm]{Proposition}
\theoremstyle{definition}
\newtheorem{defn}[thm]{Definition}

\theoremstyle{remark}
\newtheorem{rem}[thm]{Remark}

\numberwithin{equation}{section}

\newcommand{\Z}{\mathbf{Z}}

\newcommand{\F}{\textnormal{F}}
\newcommand{\R}{\mathbf{R}}
\newcommand{\C}{\mathcal{C}}

\newcommand{\HHH}{\mathbf{H}}

\newcommand{\Homeo}{\text{Homeo}}

\newcommand{\eps}{\varepsilon}

\newcommand{\SL}{\textnormal{SL}}
\newcommand{\SO}{\textnormal{S0}}

\newcommand{\Ind}{\text{Ind}}

\newcommand{\HHHH}{\text{H}}

\title[Central extension of surface groups]{Integrable measure equivalence and the central extension of surface groups}

\date{}
\begin{document}
\author[Das]{Kajal Das}
\address{UMPA UMR 5669, ENS Lyon\\  69364 Lyon cedex 7\\FRANCE}
\email{kajal.das@ens-lyon.fr}

\author[Tessera]{Romain Tessera}
\address{Laboratoire de Math\'ematiques\\
B\^atiment 425, Universit\'e Paris-Sud 11\\
91405 Orsay\\FRANCE}
\email{romain.tessera@math.u-psud.fr}

\maketitle
\begin{abstract}
Let $\Gamma_g$ be a surface group of genus $g\geq 2$. It is known that
the canonical central extension $\tilde{\Gamma}_g$  and the direct product $\Gamma_g\times \Z$ are quasi-isometric.
It is also easy to see that they are measure equivalent. By contrast, in this paper, we prove that quasi-isometry and measure equivalence cannot be
achieved ``in a compatible way". More precisely, these two groups are not uniform (nor even integrable) measure equivalent.
In particular, they cannot act continuously, properly and cocompactly by isometries on the same proper metric space, or equivalently they are not uniform
lattices in a same locally compact group. 
\end{abstract}
\section{Introduction}

\textit{Measure equivalence} (ME) is an equivalence relation on finitely generated groups introduced by Gromov in \cite{Gr}, as a measure-theoretic analogue of 
quasi-isometry (QI). The first detailled study of ME was performed in the work of Furman \cite{Fu} in the context of ME-rigidity of lattices in higher rank 
simple Lie groups. \textit{$L^p$-measure equivalence} ($L^p$-ME) is defined by imposing $L^p$-condition 
on the \textit{cocycle maps} arising from a measure equivalence relation. Such integrability condition is implicit in Margulis's proof of the normal subgroup theorem for irreducible lattices \cite{Ma}. It plays a prominent role in the work of Shalom \cite{Sh1}, where the $L^2$-integrability condition on the cocycle maps is used for inducing 1-cocycles associated to certain non-uniform lattices 
to 1-cocycles of their ambient groups. Shalom also introduces the concept of \textit{uniform measure equivalence} (UME) or
$L^{\infty}$-ME in \cite{Sh2} where he makes the crucial observation that UME and QI coincide for amenable groups (see also  \cite{LSW}, \cite{Sa}).
The most significant achievement in the context of $L^p$-ME has been recently obtained by Bader, Furman, Sauer in \cite{BFS}. These authors prove 
an \textit{integrable measure equivalence} or $L^1$-ME-rigidity result for lattices in Isom($\HHH^n$), where $n\geq 2$. 
On the side of amenable groups, Austin has shown that virtually nilpotent groups which are $L^1$-ME have bi-Lipschitz equivalent asymptotic cones \cite{T}. In an appendix of that paper, Bowen proves that for general finitely generated groups, the growth function is invariant under $L^1$-ME. 

It clearly appears from this already impressive list of results that $L^p$-measure equivalence is becoming a central notion, lying at the intersection of {\it measured} and {\it geometric} group theories.

We prove in the present paper that the canonical central
extension of the surface group $\Gamma_g$ of genus $g\geq 2$ and its direct product with $\Z$ are not $L^1$-measure equivalent, although they are known to be measure equivalent {\it and}
quasi-isometric. 
Before stating more precise results, let us recall some background.

\subsection{Central extension of surface groups}
Let $g\geq 2$, and $\Gamma_g$ be the fundamental group of the compact orientable surface of genus $g$. 
Recall its classical presentation $$\langle x_1,\ldots x_g, y_1,\ldots, y_g; [x_1,y_1]\ldots [x_g,y_g]\rangle.$$ 
Denote by $R= [x_1,y_1]\ldots [x_g,y_g]$. 
Let $\tilde{\Gamma}_g$ be the central extension of $\Gamma_g$ given by the presentation
$$\langle z, x_1,\ldots x_g, y_1,\ldots, y_g; Rz^{-1}, \textnormal{z is central} \rangle.$$ 
Clearly, this central extension is such that its center, generated by $z$, is contained in the derived subgroup of $\tilde{\Gamma}_g$.
Another way to describe this extension is as follows. Given an inclusion of $\Gamma_g\hookrightarrow\SL(2,\R) $ as a cocompact lattice, by a well-known result of Milnor  \cite{Mi},
$\tilde{\Gamma}_g$ is isomorphic to the pre-image of $\Gamma_g$ in the universal cover $\tilde{\SL}(2,\R)$ of $\SL(2,\R)$. 

\subsection{Quasi-isometry versus measure equivalence}
\hspace{5.5mm}It is well-known that \\ $\tilde{\SL}(2,\R)$ and $\SL(2,\R)\times\Z$ are quasi-isometric, from where it follows that 
$\tilde{\Gamma}_g$  and $\Gamma_g\times \Z$ are themselves quasi-isometric.
Let us briefly recall the simple argument: Let $T$ be the subgroup of upper triangular matrices in $\SL(2,\R)$. 
It is a closed cocompact subgroup, therefore quasi-isometric to $\SL(2,\R)$. 
On the other hand being simply connected, its pre-image in $\tilde{\SL}(2,\R)$ is a direct product with $\Z$. 

Besides, $\tilde{\Gamma}_g$  and $\Gamma_g\times \Z$ are measure equivalent. 
Indeed, this follows from the fact that $\tilde{\SL}(2,\R)$ has a lattice obtained by pulling back a free lattice in $\SL(2,\R)$
(observe that a central extension of a free group always splits). 

By contrast, we shall see that quasi-isometry and measure equivalence cannot be achieved ``in a compatible way".
\subsection{Integrable measure equivalence}

Given countable discrete groups $\Gamma$ and $\Lambda$, a measure equivalence (ME) coupling between them is a nonzero $\sigma$-finite measure space $(X,\mu)$, which admits commuting $\mu$-preserving actions of  $\Gamma$ and $\Lambda$  which both have finite-measure fundamental domains, respectively  $X_{\Gamma}$ and $X_{\Lambda}$. Let $\alpha:\Gamma\times X_\Lambda\rightarrow \Lambda$  
 (resp.\ $\beta:\Lambda\times X_\Gamma\rightarrow \Gamma$) be the corresponding cocycle defined by the rule: for all $x\in X_{\Lambda}$, and all 
 $\gamma\in \Gamma$, $\alpha(\gamma, x)\gamma x \in X_\Lambda$ (and symmetrically for $\beta$). If,  for any $\lambda\in \Lambda$ and 
 $\gamma\in \Gamma$, the integrals 
 $$\int_{X_\Lambda}|\alpha(\gamma,x)|^p d\mu(x) \hspace*{.5cm} \text{and} \hspace*{.5cm} \int_{X_\Gamma}|\beta(\gamma,x')|^pd\mu(x')$$
 are finite, then the coupling is called $L^p$-ME and the groups are called $L^p$-measure equivalent. The strongest form is when 
 $p=\infty$, in which case the coupling is called {\it uniform}, and the groups {\it uniformly measure equivalent (UME)}, as it generalizes the case of 
 two uniform lattices in a same locally compact group. For $p=1$, the coupling is called {\it integrable}, and the groups are said to be 
 {\it integrable measure equivalent (IME)}.
 
\subsection{Main results}

The main goal of this paper is to prove the following theorem.

\begin{thm}\label{thm:notUME}
The groups $\tilde{\Gamma}_g$ and $\Gamma_g\times \Z$ are not IME (therefore not $L^p$-measure equivalent for $1\leq p\leq \infty$).
\end{thm}

\begin{cor}
The groups $\tilde{\Gamma}_g$ and $\Gamma_g\times \Z$ are not uniform lattices in a same locally compact group.
\end{cor} 
As already mentioned,  $\tilde{\Gamma}_g$ and $\Gamma_g\times \Z$ are ME. This can be strengthened as follows, showing that Theorem \ref {thm:notUME} is optimal in a strong sense:

\begin{thm}\label{prop:p<1}
The groups $\tilde{\Gamma}_g$ and $\Gamma_g\times \Z$ admit an ME coupling which is in $L^p$ for all $p<1$.
\end{thm}
The last result was suggested to us by Shalom. Its proof, which is given in \S \ref{section:p<1}, relies on the fact, proved in \cite{Sh1}, that the standard ME-coupling between $\Gamma_g$ and a free lattice in $\SL(2,\R)$ is in $L^p$ for all $p<1$ (this extends to their preimages in $\tilde{\SL}(2,\R)$, see Proposition \ref {prop:Lp<1}). To apply this to our situation it remains to establish transitivity of the relation ``having an ME coupling which is in $L^p$ for all $p<1$". This is obtained by slightly modifying the proof of \cite{BFS} that $L^p$-measure equivalence is transitive when $p\geq 1$.

We now proceed with a description of some intermediate steps in the proof of Theorem \ref {thm:notUME} which we believe to be of independent interest.

\subsection{An ergodic theorem for integrable cocycles}
In order to prove Theorem \ref{thm:notUME}, one needs to be able to distinguish non-trivial central extensions from trivial ones from the ``ergodic point of view". This is done through the following result.

\begin{prop}\label{prop:growthcocycle}
Let $$1\to C\to \tilde{G}\to G\to 1 $$
be a central extension such that $C$ is isomorphic to $\Z$ and contained in the derived subgroup of $\tilde{G}$. We assume $\tilde{G}$ finitely generated
and equipped with a word metric $|\cdot|_{\tilde{G}}$. Let $(\Omega,\nu)$ be a standard probability space on which $\tilde{G}$ acts by measure-preserving 
automorphisms. Then, for every $1\leq p<\infty$, every $1$-cocycle with values in $L^p(\Omega,\nu)$ is sublinear in restriction to the central subgroup $C$:
$$\frac{\|b(c)\|}{|c|_{\tilde{G}}}\to 0$$
as $c\in C$ and $|c|_{\tilde{G}}\to \infty$, for all $b\in Z^1(\tilde{G},\pi)$, where $\pi$ is the norm-preserving representation of $\tilde{G}$ on $L^p(\Omega,\nu).$
\end{prop}
This proposition is proved in Section \S \ref{sec:sublinear}.
The idea behind this statement goes back to Shalom's proof that Property $H_T$ is stable under central extension \cite{Sh3} (see also \cite{Sh2}),
and culminates in a recent paper of Bader, Rosendal and Sauer \cite{BRS}, where very general results are obtained under optimal assumptions
(see \S \ref{section:generalization} for more details). In \cite{ANT}, a short proof of Proposition \ref{prop:growthcocycle} is given 
for the particular case of the Heisenberg
group. In the present paper, we essentially reproduce this proof which is based on the Mean Ergodic Theorem.
Applying the same ideas, one also obtains a proof of Serre's stability of 
Property FH under central extensions which extends to super-reflexive Banach spaces (see \S \ref {section:FH}).

\subsection{Monod-Shalom's ME-rigidity}

A crucial technical step for the proof of Theorem \ref{thm:notUME} is the following statement. 
\begin{thm}\label{thm:mecenter}\cite{MS}
Suppose $(X,\mu)$ is an ME-coupling for $\Gamma:=\tilde{\Gamma}_g$ and $\Lambda:=\Gamma_g\times \Z$. Then there exists a fundamental domain $X_\Lambda$
of $\Lambda$ so that the associated cocycle $\alpha(x,\cdot)$ sends the center of $\Gamma$ to the center of $\Lambda$ for almost all
$x\in X_\Lambda$. 
\end{thm}

In \cite{MS}, the above result is hidden in the proof of Theorem 1.17  where it is shown for the central extensions of groups having nonzero second bounded
cohomology, with coefficients in some $C_0$-representation on some separable Hilbert space. For the convenience of the reader, we sketch its proof 
for the particular case of surface groups in \S \ref{section:proofMain}, following the arguments given in \cite{MS}.

As Yehuda Shalom pointed to us, as stated, Theorem \ref {thm:mecenter} is {\it a priori} not enough for our purpose, since changing the fundamental domain might result in a cocycle which is no longer integrable. However, in our very specific situation (see \S \ref{section:proofMain}) the proof of Theorem \ref {thm:mecenter} provides some additional information that enables us to go around this difficulty. 

\subsection{Organization} Section \ref{sec:bounded} is dedicated to the proof of  Theorem \ref{thm:mecenter}. The proof of Proposition \ref{prop:growthcocycle} is given in \S \ref{sec:ergodic}.  Theorem \ref{thm:notUME} is proved in \S \ref {section:proofMain}. In \S \ref{sec:p<1}, we establish Theorem \ref {prop:p<1}. Finally, \S \ref{sec:further} deals with further results in the spirit of Proposition \ref {prop:growthcocycle}.

\subsection{Acknowledgements}
We thank Yehuda Shalom for signaling a point that we had overlooked in an earlier version of this paper and for suggesting Theorem \ref{prop:p<1}.
We are grateful to Bachir Bekka for helping us with the bibliography.

\section{Measure Equivalence and Bounded Cohomology}\label{sec:bounded}

In preparation for the proof of Theorem \ref{thm:notUME} and Theorem \ref{thm:mecenter}, we first recall some terminology and define some 
notation regarding ME-coupling and bounded cohomology. We refer to \cite{Gr}, \cite{MS}, \cite{Sh3} for details. 

Following the notation used in the introduction, let $\Gamma$ and $\Lambda$ be two countable discrete groups which are ME, and let $(X,\mu)$ be a coupling space.
We let $X_\Lambda$ be a fundamental domain for $\Lambda$  and $\alpha:\Gamma\times X_\Lambda\rightarrow \Lambda$ be the corresponding
cocycle map defined by the rule that $\alpha(\gamma, x)\gamma x \in X_\Lambda$. We shall denote the element $\alpha(\gamma, x)\gamma x$ by $\gamma\cdot x$.

Given some isometric representation $\pi$ of $\Lambda$ on a Banach space $B$, we define the induced representation $\Ind_\Lambda^\Gamma\pi$ of $\Gamma$ on
$$L^p(X_\Lambda,B):=\{\psi:X_\Lambda\rightarrow B \hspace*{.2cm}  |  \int_{X_\Lambda}|\psi|^p d\mu<\infty\} $$
in the following way:
$$\gamma\psi(x)=\pi(\alpha(\gamma^{-1},x)^{-1})\psi(\gamma^{-1}\cdotp x).$$

Let us now recall the concept of bounded cohomology of a discrete group with coefficients in a
representation on some separable Banach space (see \cite{MS} for details). Suppose $(\pi, E)$-is a $\Gamma$-module such that 
$E$ is the dual of some separable Banach space and $\Gamma$ action is defined by the adjoint actions. $(\pi, E)$ is called \textit{a coefficient 
$\Gamma$-module}. The \textit{bounded cohomology} of $\Gamma$ with coefficients in $(\pi, E)$, denoted by $\HHHH_b^{\bullet}(\Gamma,\pi)$, is defined as the cohomology of the complex
$$0\rightarrow l^\infty(\Gamma, E)^\Gamma\rightarrow l^\infty(\Gamma^2, E)^\Gamma\rightarrow l^\infty(\Gamma^3, E)^\Gamma\rightarrow\cdots,$$
where the $\Gamma$-action is defined on $l^\infty(\Gamma^n, E)$ in the following way:
$$(\gamma\cdot f)(\gamma_0,\dots,\gamma_n)=\pi(\gamma)(f(\gamma^{-1}\gamma_0,\dots,\gamma^{-1}\gamma_n)).$$

\bigskip

\subsection{Proof of Theorem \ref{thm:mecenter}}\label{pfmecenter}
Recall that with our notation, $\Lambda=\Z\times \Gamma_g$, while $\Gamma$ is the non-trivial central extension $\tilde{\Gamma_g}$.

We know that $\Gamma_g$ embeds as a cocompact lattice inside $G=SL_2(\R)$. Consider the quasi-regular representation $\Gamma_g\curvearrowright_\pi L^2(G/\Gamma_g)$ which splits as $\textbf{1}\oplus\mathcal{H}_0$,  
where $\mathcal{H}_0$ denotes the orthogonal complement of the constant functions. 
We denote the representation of $\Gamma_g$ on $\mathcal{H}_0$ by $\pi_0$. It follows from  a theorem of Howe and Moore \cite{HM} that $\pi_0$ is $c_0$ (i.e.\ its coefficients vanish at infinity). 

From Proposition 7.12 of \cite{MS}, it follows that  $\HHHH_b^2(\Gamma_g,\pi_0)\neq 0$.
We extend $\pi_0$ to a representation  (still denoted by $\pi_0$) of  $\Lambda=\Gamma_g\times \Z$ by letting $\Z$ act trivially.  By Corollary 3.6 of \cite{MS}, the \textit{inflation map} sends $\HHHH_b^2(\Gamma_g,\pi_0)$ 
injectively inside $\HHHH_b^2(\Lambda,\pi_0)$. Now, we induce this representation on $\Gamma$. By Theorem 4.4 of \cite{MS}, the \textit{induction map} from $\HHHH_b^2(\Lambda,\pi_0)$ to 
$\HHHH_b^2(\Gamma,\Ind_\Lambda^\Gamma\pi_0)$ is injective. These two facts together imply that $\HHHH_b^2(\Gamma,\Ind_\Lambda^\Gamma\pi_0)$ is nonzero.
By Proposition 3.8 (\cite{MS}), the \textit{inflation map} from $\HHHH_b^2(\Gamma/\Z,(\Ind_\Lambda^\Gamma\pi_0)^\Z)$ to $\HHHH_b^2(\Gamma,\Ind_\Lambda^\Gamma\pi_0)$ is an isomorphism (which
is due to the fact that $\Z$ is a normal amenable subgroup of $\Gamma$).
Since $\HHHH_b^2(\Gamma,\Ind_\Lambda^\Gamma\pi_0)$ is nonzero, we obtain that there exists a nonzero $\Z$-invariant vector in $\Ind_\Lambda^\Gamma\pi_0$. 
This means that there exists a nonzero measurable function
$\psi:X_\Lambda\rightarrow \mathcal{H}_0$ such that 
$$\psi(\gamma\cdot x)=\pi_0(\alpha(\gamma,x))\psi(x)$$
for all $\gamma$ in $\Z$ and for almost all $x$ in $X_\Lambda$. Consider the quotient space $\pi_0(\Lambda)\backslash\mathcal{H}_0=\pi_0(\Gamma_g)\backslash\mathcal{H}_0$. Being $c_0$,  the action
of $\Gamma_g$ on $\mathcal{H}_0$ is smooth, so we can get a measurable 
section $s$ from $\Gamma_g\backslash\mathcal{H}_0$ to $\mathcal{H}_0$, and  a measurable map $f:X_\Lambda\rightarrow \Gamma_g$ satisfying 
$$\pi_0(f(x))\psi(x)=s([\psi(x)]),$$
for almost all $x\in X_\Lambda$, where $[\xi]$ denotes the image of $\xi\in\mathcal{H}_0$ in the quotient space $\Lambda\backslash\mathcal{H}_0$.
By the definition of $f$, $f(\gamma\cdot x)\alpha(\gamma,x)f(x)^{-1}$ fixes $s([\psi(x)])$ for almost all $x\in X_\Lambda$. 
But, the stabilizer of each nonzero vector in $\mathcal{H}_0$ is $\Z$. 
Therefore, modifying $f$ on $\{x | \psi(x)=0\}$, if needed, we get $f(\gamma\cdot x)\alpha(\gamma,x)f(x)^{-1}\in \Z$ for all $\gamma\in\Z$ and for almost all $x\in X_\Lambda$. 
We define the new fundamental domain $X_\Lambda'=\{f(x)x : x\in X_\Lambda\}$. 
Now, we have  
$$\gamma f(x)x=f(x)\gamma x=f(x) \alpha(\gamma,x)^{-1}\gamma\cdot x=\left(f(\gamma\cdot x)\alpha(\gamma,x)f(x)^{-1}\right)^{-1}\left(f(\gamma\cdot x)\gamma\cdot x\right)$$
for all $\gamma\in\Z$ and for all $x\in X_\Lambda$. It follows that the cocycle $\alpha'$ defined by the formula $$\alpha'(\gamma,f(x)x)=f(\gamma\cdot x)\alpha(\gamma,x)f(x)^{-1},$$ sends the center of $\Gamma$ inside the center of $\Lambda$ for almost all $y\in X_\Lambda'$.
\hfill $\square$


\section{Reduced cohomology and central extension}\label {sec:ergodic}
\subsection{Proof of Proposition \ref{prop:growthcocycle}}\label{sec:sublinear}
We start with the following simple observation.
\begin{lem}\label{lem:afpoints}
Let $G$ be a locally compact, compactly generated group acting by isometries on a metric space $X$, and let $|\cdot|$ be the word metric on $G$ associated
to some compact generating subset $S$. Assume that this action has almost-fixed points, \i.e, for all $\eps>0$, there exists $x\in X$ such that 
$\sup_{s\in S} d(sx,x)\leq \eps$. Then its orbits are sublinear, in the sense that
$$\frac{d(gx,x)}{|g|}\to 0$$
for every $x\in X$, as $|g|\to \infty$. 
\end{lem}
\noindent{\it Proof.}
Let $\lambda=\limsup_{|g|\to \infty} \frac{d(gx,x)}{|g|}$.
Clearly, $\lambda$ does not depend on $x$. Applying it to $x$ such that $\sup_{s\in S} d(sx,x)\leq \eps$, we see that it is less than $\eps$ for any 
$\eps>0$, hence equal to $0$. Indeed, write $g$ as a product of $|g|$ elements in $S$, $g=s_1s_2\ldots$ and use triangular inequality to write
$$d(gx,x)\leq d(s_1x,x)+d(s_1s_2x,s_1x)+\ldots=d(s_1x,x)+d(s_2x,x)+\ldots\leq \eps|g|. \; \qed$$ 

The following theorem is originally due to Alaoglu and Birkhoff for super-reflexive Banach spaces. In \cite{BRS}, the authors introduce the following 
terminology: a (strongly) continuous representation of a locally compact group is weakly almost periodic (wap) if its orbits $\pi(G)v$ are weakly 
relatively compact.
\begin{thm}\label{thm:decomposition}\cite{AB, BFGM, BRS}
Let $(G,\pi)$ be a wap representation of a group $G$ on a Banach space $B$. Then the space of $\pi(G)$-invariant vectors has a canonical complement. 
In particular, this complement is invariant under the group of all norm-preserving linear transformations of $B$. 
\end{thm}
This includes, for instance, the case where $B$ is reflexive, but also the case where $G$ acts on $L^1$ of a probability space via measure-preserving
automorphisms. In the case of a single transformation  (i.e.\ $G=\Z$), the theorem is a consequence of the Mean Ergodic Theorem which was proved for
wap representations already in 1938 \cite{Y} (see also \cite{K}).

\begin{prop}\label{prop:Meanergodic}
Let  $$1\to C\to \tilde{G}\to G\to 1 $$
be a central extension of locally compact groups such that $C\simeq \Z$ and let $(\tilde{G},\pi)$ be a continuous representation of $\tilde{G}$ on a Banach space $B$
without $\pi(C)$-invariant vectors. Suppose, in addition, that $\pi(C)$ satisfies the Mean Ergodic Theorem. Then $\overline{H}^1(\tilde{G},\pi)=0$.
\end{prop}
This proposition is a special case of \cite[Theorem 2]{BRS} when the representation is wap. Their conclusion is stronger as they also obtain vanishing of
the reduced cohomology groups of higher degree. Their assumptions on $\tilde{G}$ and $C$ are also more general, but we shall see below that our proof can be
easily extended to this situation.

\begin{proof}
Let $\sigma$ be the affine action associated to $b$. 
Let $c$ be a generator of $C$ and let $$v_n=\frac{1}{n}\sum_{k=1}^nb(c^k).$$ For every $g\in \tilde{G}$, one has
\begin{eqnarray*}
\sigma(g)v_n-v_n & = & \frac{1}{n}\sum_{k=1}^n(b(gc^k)-b(c^k))\\
                            & = & \frac{1}{n}\sum_{k=1}^n(b(c^kg)-b(c^k))\\
                          & = & \frac{1}{n}\sum_{k=1}^n\pi(c)^kb(g) \\
\end{eqnarray*}
which tends to zero when $n\to \infty$ by the Mean Ergodic Theorem.\end{proof}

\

Proposition \ref{prop:growthcocycle} is a corollary of the following more general statement:

\begin{prop}\label{prop:growthcocycleBis}
Let $$1\to C\to \tilde{G}\to G\to 1 $$
be a central extension such that $C$ is isomorphic to $\Z$ and contained in the derived subgroup of $\tilde{G}$. We assume $\tilde{G}$ finitely generated
and equipped with a word metric $|\cdot|_{\tilde{G}}$. Let $(B,\pi)$ be a wap representation of $\tilde{G}$. Then,  every $1$-cocycle $b\in Z^1(\tilde{G},\pi)$ is sublinear in restriction to the central subgroup $C$:
$$\frac{\|b(c)\|}{|c|_{\tilde{G}}}\to 0$$
as $c\in C$ and $|c|_{\tilde{G}}\to \infty$.
\end{prop}

\begin{proof}
We first apply Theorem \ref{thm:decomposition} to reduce to the case where $\pi(\tilde{G})$ has no nonzero invariant vectors. Indeed, otherwise $b$ decomposes
accordingly as $b'+b''$, where $b'$ is a morphism and therefore factors through $\tilde{G}/C$ since $C$ belongs to $[\tilde{G},\tilde{G}]$. By the Mean Ergodic Theorem, $B$
decomposes canonically as a direct sum $B_1\oplus B_2$, where $B_1$ is the space of $\pi(C)$-invariant vectors. Any $1$-cocycle $b\in Z^1(\tilde{G},\pi)$
decomposes accordingly as a direct sum $b_1+b_2$. Observe that for all $c\in C$, $b_1(c)$ is $\pi(\tilde{G})$-invariant. Hence $b_1$ is trivial in restriction to $C$.
Now, by Proposition \ref{prop:Meanergodic}, the affine action associated to $b_2$ has almost fixed points. We conclude with Lemma \ref{lem:afpoints}.
\end{proof}

\section{Proof of  Theorem \ref{thm:notUME}}\label{section:proofMain}

 From now on, we let $X$ be an $L^1$-coupling between the groups $\Gamma:=\tilde{\Gamma}_g$ and 
 $\Lambda:= \Gamma_g\times \Z$, and $X_\Gamma$ and $X_\Lambda$ be fundamental domains of $\Gamma$ and $\Lambda$ respectively.
 By Theorem \ref{thm:mecenter}, there exists a fundamental domain $X_\Lambda'$ so that the resulting cocycle $\alpha'(\cdot, x)$ sends the center of $\Gamma$ to 
 the center of $\Lambda$ for almost all $x\in X_\Lambda'$. Moreover, the proof of that theorem specifies $X_\Lambda'$ as $\{f(x)x,\;x\in X_\Lambda\}$ for some measurable function $f:X_\Lambda\to \Gamma_g$, while $\alpha':\Gamma\times X_{\Lambda}'\to \Lambda$ is defined as 
 \begin{equation}\label {eq:alpha'}
 \alpha'(\gamma,f(x)x)=f(\gamma\cdot x)\alpha(\gamma,x)f(x)^{-1},
 \end{equation}
for all $x\in X_{\Lambda}.$
   Observe that $\Lambda$ has an obvious morphism $b$ to $\R$, mapping its second factor to $\Z\subset \R$. 
 This morphism can be interpreted as a $1$-cocycle $b\in Z^1(\Lambda,1)$ associated to the trivial representation. 
 Obviously, this cocycle grows linearly in the direction of $\Z$. 
 
We induce the cocycle $b\in Z^1(\Lambda,1)$ to a 1-cocycle $B$ of the induced representation $\Ind_\Lambda^\Gamma1$ by the following expression:
$$B(\gamma)(x):=b(\alpha(\gamma, x)). $$
This formula makes sense even if $X$ is simply an ME coupling. However, its $L^1$-integrability, and therefore the fact that $B\in Z^1(\Gamma,\Ind_\Lambda^\Gamma1)$ follows from the condition that 
the coupling is integrable. Now, the fact that $b$ factors through $\Z$ together with (\ref{eq:alpha'}) implies that for all $x\in X_{\Lambda}$ and $\gamma\in \Gamma$,
\begin{equation}\label{eq:indcocycle}
B(\gamma)(x)=B'(\gamma)(y):=b( \alpha'(\gamma,y)),
\end{equation}
where $y=f(x)x\in X_{\Lambda}'$. 
It follows that $B'\in Z^1(\Gamma,\Ind_\Lambda^\Gamma1)$ and satisfies $\|B'(\gamma)\|=\|B(\gamma)\|$ for all $\gamma\in \Gamma$.

We shall prove that $B'$ does not grow sublinearly in the direction of the central subgroup $C<\tilde{\Gamma}_g$, contradicting Proposition 
\ref{prop:growthcocycle}. We denote the set of integers between $a$ and $b$ in the center of $\Gamma$ by the symbol $[a,b]_\Gamma$. 
Similarly, we define $[a,b]_\Lambda$. In the rest of the proof, $\gamma$ and $\lambda$ denote elements in the center of $\Gamma$ and $\Lambda$,
 respectively. Without loss of generality, we can assume that $\mu(X_\Lambda'\cap X_\Gamma)>0$.
Therefore, for every positive integer $k$, 
$$\mu([-kn,kn]_\Gamma X_\Lambda')\geq \mu(X_\Lambda'\cap X_\Gamma) (2kn+1),$$
which implies that 
$$\mu([-kn, kn]_\Gamma X_\Lambda' \backslash [-n, n]_\Lambda X_\Lambda') \geq \mu(X_\Lambda'\cap X_\Gamma)(2kn+1) - \mu(X_\Lambda') (2n+1).$$ 
So, for  $k$ large enough,
\begin{equation}\label{eq:ordfd}
\mu([-kn, kn]_\Gamma X_\Lambda' \backslash [-n, n]_\Lambda X_\Lambda')\geq n.
\end{equation}
Now, we have
\begin{align*}
\frac{1}{2kn}\sum_{i=-kn}^{kn}\|B'(i)\|
&= \frac{1}{2kn}\sum_{\gamma=-kn}^{kn}\int_{X_\Lambda'}|\alpha'(\gamma,x)| d\mu(x)\\
&= \frac{1}{2kn}\sum_{\gamma=-kn}^{kn}\sum_{\lambda=-\infty}^{\infty} |\lambda| \mu(\lambda X_\Lambda' \cap \gamma X_\Lambda')\\
&\geq \frac{1}{2kn}\sum_{\gamma=-kn}^{kn}\sum_{|\lambda|> n} |\lambda| \mu(\lambda X_\Lambda' \cap \gamma X_\Lambda')\\
& > \frac{1}{2k}\mu([-kn, kn]_\Gamma X_\Lambda' \backslash [-n, n]_\Lambda X_\Lambda').
\end{align*}

Therefore, by using \ref{eq:ordfd}, we get $\frac{1}{2kn}\sum_{i=-kn}^{kn}\|B'(i)\|\geq n$, which finishes the proof of the theorem.

\section{$L^p$-measure equivalence for $p<1$}\label {section:p<1}\label{sec:p<1}

This section is dedicated to the proof of Theorem \ref{prop:p<1}. 
Let $G$ be a locally compact group equipped with a Haar measure $\mu$, and let $\Lambda$ and $\Gamma$ be two lattices in $G$. Consider the measure-preserving action of $\Lambda$ (resp.\ $\Gamma$) by left (resp.\ right) translation on $(G,\mu)$, and let $X_\Lambda$ (resp.\ $X_\Gamma$) be a fundamental domain. This defines an ME coupling between these groups.

\begin{prop}\label {prop:Lp<1}
Two lattices $\Lambda$ and $\Gamma$ in $\SL(2,\R)$ admit fundamental domains such that the corresponding  cocycles are in $L^p$ for all $p<1$.  The same holds for the pull back $\tilde{\Lambda}$ and $\tilde{\Gamma}$ in $\tilde{\SL}(2,\R)$.
\end{prop}
\begin{proof}
The first statement follows from the proof of \cite[Theorem 3.8]{Sh1}. 
The second statement  relies on the fact that the central extension $\tilde{\SL}(2,\R)$ of $\SL(2,\R)$ can be represented by a {\it bounded} $2$-cocycle: one can consult for instance \cite{Gh}, but for the convenience of the reader, let us briefly sketch its proof. 

Recall that the action of $\SL(2,\R)$ on the boundary of the hyperbolic plane induces an embedding from $\SL(2,\R)$ to $\Homeo_+(S^1)$. The restriction of this embedding to $\SO(2,\R)$ being a homotopy equivalence \cite[Proposition 4.2]{Gh}, we deduce that the fundamental group of $\Homeo_+(S^1)$ is isomorphic to $\Z$, so that we have the following central extension
$$1\to \Z\to \tilde{\Homeo}_+(S^1)\to \Homeo_+(S^1)\to 1,$$
where $\tilde{\Homeo}_+(S^1)$ is naturally identified to a subgroup of $ \Homeo_+(\R)$, and $\Z$ as the subgroup of integral translations.
Denote the projection by $p$, and consider the section 
$\sigma:\Homeo_+(S^1)\to \tilde{\Homeo}_+(S^1)$ defined as follows: given $f\in \Homeo_+(S^1)$, $\sigma(f)$ is the unique preimage of $f$ under $p$ such that $\sigma(f)(0)\in [0,1)$.  
One easily checks that the $2$-cocycle $c(f,f')=\sigma(f)\sigma(f')\sigma(ff')^{-1}\in \Z$ is bounded, and more precisely that it takes values in $\{0,1\}$ (see for instance \cite[Lemma 6.3]{Gh}). By restriction, $\sigma$ defines a section for the exact sequence
$$1\to \Z\to \tilde{\SL}(2,\R)\to \SL(2,\R)\to 1.$$
Now let $S_{\Lambda}$ be a finite symmetric generating subset of $\Lambda$. Note that the set $\sigma(S_{\Lambda})\cup\{\eps\}$, where $\eps$ is a generator of the center, generates $\tilde{\Lambda}$. We shall denote by $|\cdot|_{\Lambda}$ and $|\cdot|_{\tilde{\Lambda}}$ the corresponding word lengths. 

\begin{lem}\label{lem:length}
For all $\lambda\in \Lambda$, one has $|\sigma(\lambda)|_{\tilde{\Lambda}}\leq 2|\lambda|_{\Lambda}$. 
\end{lem}

\begin{proof}
Let $n=|\lambda |_{\Lambda}$, so that $\gamma=s_1\ldots s_n$, where for every $i$, $s_i$ lies in $S_{\Lambda}$. The fact that the cocycle $c$ takes values in $\{0,1\}$ implies that $\sigma(s_1\ldots s_n)$ differs from $\sigma(s_1)\ldots,\sigma(s_n)$ by an element  of the center of absolute value at most $n$. Hence the lemma follows. \end{proof}
The second statement of Proposition \ref{prop:Lp<1} now results from the following  lemma. 
\end{proof}

\begin{lem}\label{lem:liftcocycle}
We keep the notation of Proposition \ref{prop:Lp<1}.
Let $X_\Lambda$ be a fundamental domain for the action of $\Lambda$ on $\SL(2,\R)$, and let $\alpha: \Gamma\times X_\Lambda\to \Lambda$ be the associated cocycle. Then $\sigma(X_\Lambda)\subset \tilde{\SL}(2,\R)$ is a fundamental domain for the action of $\tilde{\Lambda}$ and the corresponding cocycle $\tilde{\alpha}:\tilde{\Gamma}\times \sigma(X_\Lambda)\to \tilde{\Lambda}$ satisfies that for all $\gamma\in \tilde{\Gamma}$, there exists a constant $C=C(\gamma)$ such that for all $x\in X_\Lambda$, 
$$|\tilde{\alpha}(\gamma,\sigma(x))|_{\tilde{\Lambda}}\leq 2|\alpha(p(\gamma),x)|_{\Lambda}+C.$$
\end{lem}
\begin{proof}
We let $\gamma\in \tilde{\Gamma}$, and let $z_\lambda\in \Z$ be such that $\gamma^{-1}=(\sigma\circ p(\gamma))^{-1}z_\lambda$. Let $x\in X_\Lambda$: by definition, $\tilde{\alpha}(\gamma,\sigma(x))$ is the unique $\lambda\in \tilde{\Lambda}$ such that $\lambda \sigma(x)\gamma^{-1}=\sigma(y)$ for some $y\in X_\Lambda$.
Now, projecting to $\SL(2,\R)$, we obtain that 
\begin{equation}\label{eq}
p(\lambda)xp(\gamma)^{-1}=y,
 \end{equation}
 from which we deduce that 
\begin{equation}\label{eq:projec}
 \alpha(p(\gamma),x)=p(\lambda).
 \end{equation}
Applying $\sigma$ to (\ref{eq}), we get 
 $$\sigma(p(\lambda)xp(\gamma)^{-1})=\sigma(y).$$
We deduce from the fact that the cocycle $c$ takes values in $\{0,1\}$ that 
$$\sigma\circ p(\lambda)\sigma(x)(\sigma\circ p(\gamma))^{-1}=\sigma(y)z,$$
 where $|z|\leq 3$. 
Therefore, we have that $$\lambda=
\sigma\circ p(\lambda)z^{-1}z_{\gamma}^{-1},$$
 which, combined with (\ref{eq:projec}) and Lemma \ref{lem:length}, yields the conclusion of the lemma with $C=|z_{\gamma}|+3$.
\end{proof}

\begin{rem}
Observe that Lemma \ref {lem:liftcocycle} is actually valid under the following general hypotheses: $\Lambda$ and $\Gamma$ being two lattices in a locally compact group $G$ and $\tilde{\Lambda}$ and $\tilde{\Gamma}$ being their pull-back  in $\tilde{G}$, where $\tilde{G}$ is a central extension of $G$ associated to  a bounded $2$-cocycle. \end{rem}

\begin{cor}
Let $\F$ be a free lattice in $\SL(2,\R)$. Then $\Gamma_g$ and $\F$ admits an ME coupling whose cocycles are  in $L^p$ for all $p<1$ (actually one of them is in $L^{\infty}$ for obvious reasons). The same conclusion holds for $\Gamma_g\times \Z$ and $\F\times \Z$, and for  $\tilde{\Gamma_g}$ and $\tilde{\F}\simeq \F\times \Z$.
\end{cor}

For the proof of Theorem \ref{prop:p<1}, we proceed as follows: we have seen that $\tilde{\Gamma_g}$ and $F\times \Z$ (resp.\ $\Gamma\times \Z$ and $F\times \Z$) admit an ME-coupling whose cocycles are in $L^p$ for all $p<1$. In order to conclude we need to establish transitivity of this relation, which is the object of the following proposition.

\begin{prop}
The relation ``admitting an ME-coupling whose cocycle is in $L^p$ for all $p<1$" is transitive and therefore defines an equivalence relation among compactly generated locally compact unimodular groups. 
\end{prop}
\begin{proof}
Recall that this statement was proved for cocycles in $L^p$ for $p\geq 1$ in \cite[Appendix 1]{BFS}. The only part of that proof that needs to be adapted is \cite[Lemma A.1]{BFS}. This is done below. 
\end{proof}
\begin{lem}
Let $G$, $H$ and $L$ be compactly generated groups, $G\curvearrowright (X,\mu)$ and  $H\curvearrowright (Y,\nu)$
be finite measure-preserving actions and 
$\alpha:G\times X\to H$ and $\beta:H \times Y\to L$ be
$L^p$-integrable cocycles for all $p<1$. Consider $Z = X \times Y$ and the action of $G\curvearrowright Z$ defined by
$g: (x,y)\to (g\cdot x, \alpha(g, x)\cdot y).$
 Then the cocycle 
$\gamma: G\times Z\to L$ given by
$$\gamma(g,(x,y))=\beta(\alpha(g,x),y)$$
is in $L^p$ for all $p<1$.
\end{lem}
\begin{proof}
Let $p<p'<1$ and let us prove that $\gamma$ is in $L^p$. Fix word length on $G$, $H$ and $L$.
We know that $|\beta |_L$ is a subadditive cocycle, meaning that 
$$|\beta(hh',y)|_L\leq |\beta(h,h'y)|_L + |\beta(h',y)|_L,$$
and since $p'<1$, we also have 
$$|\beta(hh',y)|_L^{p'}\leq |\beta(h,h'y)|_L^{p'} + |\beta(h',y)|_L^{p'},$$
for a.e.\ $y\in Y$ and all $h,h'\in H$. 
Therefore 
$L_{p'}(h)=\int |\beta(h,y)|_L^{p'}d\nu(y)$ is a pseudo-length on $H$, and hence is $\leq C|h|_H$ for some constant $C$.

Now, using Jensen's inequality for $q=p'/p>1$, we have, for a.e.\ $x\in X$, and all $g\in G$,

$$\int |\beta(\alpha(g,x),y)|_L^{p}d\nu(y)\leq \left(\int |\beta(\alpha(g,x),y)|_L^{p'}d\nu(y)\right)^{p/p'}\leq \left(C|\alpha(g,x)|_H\right)^{p/p'}.$$
Since $p/p'<1$, it follows from the assumption on $\alpha$ that $\int|\alpha(g,x)|_H^{p/p'}d\mu(x)<\infty$. Hence we are done.
\end{proof}

\section{Further results}\label{sec:further}

\subsection{Generalization of Proposition \ref{prop:Meanergodic}}\label {section:generalization}

It turns out that the proof of Proposition \ref{prop:Meanergodic} can easily be adapted to the more general setting of \cite[Theorem 2]{BRS}, leading to 
what is essentially a geometric reformulation of their proof (i.e.\ in terms of almost fixed points). Looking back at the proof of Proposition 
\ref{prop:Meanergodic}, we observe that the construction of the sequence $(v_n)$ of almost fixed points  relied on the existence of a net of formal 
finite convex combinations  of elements of $C$, say $(\sum_{c\in C}\lambda_c^{(i)} c)_i$, such that the corresponding net of operators
$(\sum_{c\in C}\lambda_c^{(i)} \pi(c))_i$ converges in the strong operator topology to $0$. In the setting of Proposition \ref{prop:Meanergodic},
a natural choice was to take a Cesaro sum in order to apply the Mean Ergodic Theorem. 

\begin{prop}\label{prop:BRS}\cite[Theorem 2]{BRS}
 Let $N$ and $C$ be closed subgroups of  a locally compact group $G$, with $C$ lying in the centralisizer of $N$. Assume $\pi$ is a normed preserving
 representation on a Banach space such that $\pi(C)$ is wap and has no nonzero invariant vectors. Let $\sigma$ be an affine isometric action of $G$ whose linear part is $\pi$. Then  $\sigma(N)$ has almost fixed points.
\end{prop}
 
\begin{proof} Let us assume for simplicity that $G$ is discrete. If $C$ was isomorphic to $\Z$, we could apply verbatim the  proof of Proposition
\ref{prop:Meanergodic}. In replacement for the Mean Ergodic Theorem we shall use the Ryll-Nardzewski fixed point theorem \cite{RN}.
Indeed, for every $v\in B$, the closed convex hull of the $\pi(C)$-orbit of $v$ is weakly compact, and therefore contains some $\pi(C)$-invariant vector.
Since $C$ does not have nonzero invariant vectors, this implies that there is a sequence of such convex combinations converging to $0$.  
A diagonal argument implies that there exists a net $\delta^{(i)}=\sum_{c\in C}\lambda_c^{(i)} \pi(c)$ of such convex combinations such that 
$\|\delta^{(i)}v\|\to 0$ for all $v\in B$. Now, replace $\pi(c)v$ by $b(c)$ in each convex combination $\delta^{(i)}$:
this defines a net $v_i\in B$. 
This sequence is then shown to be almost $\sigma(N)$-fixed using that for all $c\in C$ and $n\in N$, $b(nc)-b(c)=b(cn)-b(c)=\pi(c)b(n)$. Indeed, for all $n\in N$, we get that
$$\sigma(n)v_i-v_i= \sum_{c\in C}\lambda_c^{(i)}(b(nc)-b(c))=\sum_{c\in C}\lambda_c^{(i)}\pi(c)b(n)=\delta^{(i)}b(n),$$
 which tends to zero in norm as $i\to \infty$. 
\end{proof}

\subsection{Fixed-point properties and central extensions}\label{section:FH}

Let us end this section with a Banachic version of Serre's theorem as announced in the introduction. 

\begin{defn}
Let $\C$ be a class of super-reflexive Banach spaces stable under ultralimits (such as $L^p$-spaces for a fixed $1<p<\infty$, or uniformly convex Banach 
spaces with modulus of convexity bounded from below). A locally compact group has Property $F\C$ if every continuous affine isometric action on some 
element of $\C$ has a fixed point. 
\end{defn}

\begin{thm} 
Let  $$1\to C\to \tilde{G}\to G\to 1 $$
be a central extension of locally compact groups such that $C\subset [ \tilde{G}, \tilde{G}]$. If $G$ has Property $F\C$, then so does $ \tilde{G}$. 
\end{thm}
\begin{proof}
Let $\pi$ a norm-preserving representation of $ \tilde{G}$ on some Banach space in $\C$, and let $b\in Z^1( \tilde{G},\pi)$. By Theorem \ref{thm:decomposition}, 
only two cases need to be considered: the case where $\pi( \tilde{G})=\{id\}$, and the case where it does not have nonzero invariant vectors. In the first case, 
$b$ is a morphism, and in particular factors through $G$: it is therefore trivial.

We can therefore assume that $\pi( \tilde{G})$ does not have nonzero invariant vectors. Again we can split the problem into two cases: 
either $\pi(C)=\{id\}$, or $\pi(C)$ does not have nonzero invariant vectors.
In the first case, $\pi$ induces a representation $\overline{\pi}$ of the quotient $G$.
Since $b(c)$ is $\pi( \tilde{G})$-invariant for all $c\in C$, $b$ is identically zero in restriction to $C$. It therefore factors through a cocycle in 
$Z^1(G,\overline{\pi})$, which is a coboundary by our assumption on $G$. It follows that $b$ itself is a coboundary.

In the second case, we deduce from \cite[Theorem 2]{BRS} (see Proposition \ref{prop:BRS}) that every continuous norm-preserving representation of $ \tilde{G}$
has trivial first reduced cohomology. We conclude (thanks to the following theorem of Gromov \cite{GrRW}): if a group admits an affine isometric action
on a Banach space without fixed points, then it admits an affine isometric action on some ultralimit of this Banach space without almost invariant points.
\end{proof}


\end{document}